\documentclass[12pt]{article}
\usepackage{booktabs}
\usepackage{caption}
\usepackage{mathrsfs}
\usepackage{amsmath}
\usepackage{amsfonts,amsthm,amssymb,mathrsfs,bbding}
\usepackage{float}
\usepackage{graphics,epstopdf}
\usepackage{graphics,multicol}
\usepackage{graphicx}
\usepackage{color}
\usepackage{enumerate}
\usepackage{caption}
\usepackage{rotating}
\usepackage{lscape}
\usepackage{longtable}
\allowdisplaybreaks[4]
\usepackage{tabularx}
\usepackage[colorlinks=true,anchorcolor=blue,filecolor=blue,linkcolor=blue,urlcolor=blue,citecolor=blue]{hyperref}
\usepackage{extarrows}
\usepackage{cite}
\usepackage{latexsym,bm}
\usepackage{mathtools}
\evensidemargin=\oddsidemargin\topmargin=-1.5cm

\newtheorem{thm}{Theorem}
\newtheorem{prob}{Problem}

\newtheorem{lem}{Lemma}[section]

\newtheorem{pro}{Proposition}

\newtheorem{conj}{Conjecture}[section]
\newtheorem{claim}{Claim}[section]

\addtocounter{section}{0}
\begin{document}
\title{Counting substructures and eigenvalues II: quadrilaterals}

\author{Bo Ning\thanks{College of Cyber Science, Nankai University, Tianjin 300350, P.R.
China. Email: bo.ning@nankai.edu.cn. Partially supported
by the NSFC grant (No.\ 11971346).}~~~
~~~Mingqing Zhai\thanks{Corresponding author. 
School of Mathematics and Finance, Chuzhou
University, Chuzhou, Anhui 239012, P.R. China. 
Email: mqzhai@chzu.edu.cn. Partially supported 
by the NSFC grant (No. 12171066) and APNSF (No. 2108085MA13).}}

\maketitle
{\flushleft\large\bf Abstract:}
Let $G$ be a graph and $\lambda(G)$ be the spectral radius
of $G$. A previous result due to Nikiforov [Linear Algebra 
Appl., 2009] in spectral graph theory asserted that every 
graph $G$ on $m\geq 10$ edges contains a 4-cycle 
if $\lambda(G)>\sqrt{m}$. Define $f(m)$ to be the 
minimum number of copies of 4-cycles
in such a graph. A consequence of a recent theorem due to 
Zhai et al. [European J. Combin., 2021] shows that $f(m)=\Omega(m)$.
In this article, by somewhat different techniques, we prove 
that $f(m)=\Theta(m^2)$. We left the solution to 
$\lim\limits_{m\rightarrow \infty} \frac{f(m)}{m^2}$ as a 
problem, and also mention other ones for further study.

\begin{flushleft}
\end{flushleft}

\section{Introduction}
This is the second paper of our project \cite{NZ21+}, which aims 
to study the relationship between copies of a given substructure 
and the eigenvalues of a graph. In this article, we study 
the supersaturation problem of 4-cycles under the eigenvalue condition.

The study of 4-cycles plays an important role in the history of 
extremal graph theory. The extremal number of $C_4$ (i.e., 
a 4-cycle), denoted by $ex(n,C_4)$, is defined to be the maximum
number of edges in a graph which contains no 4-cycle as a 
subgraph. The study of $ex(n,C_4)$ can be at least dated
back to Erd\H{o}s \cite{E38} eighty years ago.
A longstanding conjecture of Erd\H{o}s and Simonovits \cite{ES84}
(see also \cite[p.~84]{CG98}) states that every graph
on $n$ vertices and at least $ex(n,C_4)+1$ edges contains
at least two copies of 4-cycles when $n$ is large. Very recently,
He, Ma and Yang \cite{Y21+} announced this conjecture
does not hold for the cases $n=q^2+q+2$ where $q=4^k$ is large.

The original supersaturation problem of subgraphs in 
graphs focuses on the following function:
for a given graph $H$ and for integers $n,t\geq1$,
$$h_{H}(n,t) = min\{\#H: |V(G)|=n, |E(G)|=ex(n,H)+t\},$$
where $ex(n,H)$ is the Tur\'an function of $H$.
Establishing a conjecture of Erd\H{o}s, Lov\'asz 
and Simonovits \cite{LS83} proved that
$h_{C_3}(n,k)\geq k\lfloor\frac{n}{2}\rfloor$ for all 
$1\leq k<\lfloor\frac{n}{2}\rfloor$. But He et al.'s result 
tells us $h_{C_4}(n,1)=1$ for some positive integers $n$.
This means that supersaturation phenomenon of $C_4$ 
is quite different from the cases of triangles \cite{LS83}.
On the other hand, counting the copies of 4-cycles 
plays a heuristic important role in measuring
the quasirandom-ness of a graph (see Chap. 9 in \cite{AS08}).

As an important case of spectral Zarankiewicz problem, 
Nikiforov \cite{N10} proved that every $n$-vertex 
$C_4$-free graph satisfies that
$\lambda(G)\leq \frac{1}{2}+\sqrt{n-\frac{3}{4}}$ where $\lambda(G)$
is the spectral radius of $G$, and
the earlier bound of Babai and Guiduli \cite{BG08} 
gives the correct order of the main term.
As the counterpart of these results, we consider sufficient
eigenvalue condition (in terms of the size of a graph) for the
existence of 4-cycles. A pioneer result can be found in \cite{N09}.
\begin{thm}[\!\!\cite{N09}]\label{Thm:N09}
Let $G$ be a graph with $m$ edges, where $m\geq10$. If $\lambda(G)\geq\sqrt m$
then $G$ contains a 4-cycle, unless $G$ is a star (possibly with some isolated vertices).
\end{thm}

Recently, Theorem \ref{Thm:N09} was extended by the following.

\begin{thm}[\!\!\cite{ZLS21}]\label{Thm:ZLS21}
Let $r$ be a positive integer and $G$ be a 
graph with $m$ edges where $m\geq16r^2$. If 
$\lambda(G)\geq\sqrt m$, then $G$ contains a 
copy of $K_{2,r+1}$, unless $G$ is a star 
(possibly with some isolated vertices).
\end{thm}

Let $B_r$ be an \emph{$r$-book}, that is, the 
graph obtained from $K_{2,r}$ by adding one edge 
within the partition set of two vertices.
Very recently, Nikiforov \cite{N21} proved that, 
if $m\geq (12r)^4$ and $\lambda(G)\geq\sqrt m$, then
$G$ contains a copy of $B_{r+1}$, unless $G$ is a 
complete bipartite graph (possibly with some isolated vertices).
This result further extends above two theorems and 
solves a conjecture proposed in \cite{ZLS21}.

The central topic of this article is the following spectral radius version
of supersaturation problem
of $4$-cycles :
\begin{prob}\label{Prob:Counting4-cycles}
Let $f(m)$ be the minimum number of copies of 4-cycles over
all labelled graph $G$ on $m$ edges with $\lambda(G)>\sqrt{m}$.
Give an estimate of $f(m)$.
\end{prob}

Till now, the only counting result related to Problem \ref{Prob:Counting4-cycles}
is a consequence of Theorem \ref{Thm:ZLS21}.
Note that $K_{2,r+1}$ contains $\frac{r(r+1)}2$ 4-cycles for $r=\frac{\sqrt{m}}4.$
Theorem \ref{Thm:ZLS21} implies that $f(m)\geq \frac{m}{32}$,
unless $G$ is a star (possibly with some isolated vertices).

One may ask for the best answer to Problem \ref{Prob:Counting4-cycles}.
In this paper, we make the first progress towards this problem.

\begin{thm}\label{Thm:Mainresult}
Let $m\geq3.6\times 10^9$ be a positive integer. Then 
$f(m)\ge \frac{m^2}{2000}$, unless the graph $G$ is 
a star (possibly with some isolated vertices).
\end{thm}

Throughout the left part, we also define $f(G)$ to be the number 
of copies of 4-cycles in a graph $G$.
\begin{pro}\label{Prop}
$f(m)\leq \frac{(m-1)(m-2\sqrt{m})}8.$
\end{pro}
\begin{proof}
Let $s=\sqrt{m}+1$ and
$K_s^+$ be the graph obtained from the complete graph
$K_s$ by adding $m-{s \choose 2}$ pendent edges to one vertex of $K_s$.
Clearly, $\lambda(K_s^+)\geq \lambda(K_s)=\sqrt{m}$.
However, observe that $K_s^+$ contains ${s\choose 4}$ copies of $K_4$
and every $K_4$ contains three copies of 4-cycles.
Consequently, $f(K_s^+)=3{s\choose 4}=\frac{(m-1)(m-2\sqrt{m})}8.$
\end{proof}
Together with Theorem \ref{Thm:Mainresult} and
Proposition \ref{Prop}, one can easily find that $f(m)=\Theta(m^2)$.

Let us introduce some necessary notation and terminologies.
Let $G $ be a graph with vertex set $V(G)$ and edge set $E(G)$.
For a vertex $u\in V(G),$ we denote by $N_G(u)$ the set 
of neighbors of $u$, and by $d_G(u)$ the degree of $u$.
The symbol $G-v$ denotes the
subgraph induced by $V(G)\backslash \{v\}$ in $G$.

The paper is organized as follows. In Section \ref{Sec:Preliminaries},
we shall give some necessary preliminaries and prove a key lemma.
We present a proof of our main theorem in Section \ref{Sec:Mainproof}.
We conclude this article with one corollary and some open problems 
for further study.
\section{Preliminaries}\label{Sec:Preliminaries}

In this section, we introduce some lemmas, which will be used in
the subsequent proof.

The first lemma is known as Cauchy's Interlace Theorem.

\begin{lem}[\cite{BH12}]\label{Lem:CauchyInter}
Let $A$ be a symmetric $n\times n$
matrix and $B$ be an $r\times r$ principal submatrix of $A$ for some $r<n$. If
the eigenvalues of $A$ are $\lambda_1\geq \lambda_2\geq\cdots \geq\lambda_n$, and the eigenvalues of $B$ are
$\mu_1\geq \mu_2\geq\cdots\geq \mu_r$, then
$\lambda_i\geq \mu_i\geq \lambda_{i+n-r}$ for all $1\leq i\leq r$.
\end{lem}

The following inequality is due to Hofmeistar.

\begin{lem}[\!\!\cite{H88}]\label{lem2}
Let $G$ be a graph of order $n$ and $M(G)=\sum_{u\in V(G)}d^2_G(u)$.
Then
\begin{eqnarray}\label{eq1}
\lambda(G)\geq\sqrt{\frac1n M(G)},
\end{eqnarray}
with equality if and only if $G$ is either regular or bipartite semi-regular.
\end{lem}

\begin{lem}[\!\!\cite{LL09}]\label{lem3}
Let $G$ be a graph of order $n$ and size $m$. Then
\begin{eqnarray}\label{eq2}
f(G)=\frac18\sum_{i=1}^n\lambda_i^4+\frac m4-\frac14M(G),
\end{eqnarray}
where $\lambda_1,\ldots,\lambda_n$ are the eigenvalues of 
$G$ with $\lambda_1\geq\lambda_2\geq\cdots\geq\lambda_n.$
\end{lem}

The following result is well-known \cite{N70,BFP08}. A short 
proof can also be found in \cite{N17}.
\begin{lem}[\!\!\cite{N70,BFP08,N17}]\label{lem5}
Let $G$ be a bipartite graph with $m$ edges, where $m\geq1$. Then
$\lambda(G)\leq\sqrt{m},$
with equality if and only if $G$ is a complete bipartite 
graph (possibly with some isolated vertices).
\end{lem}

We need prove the last lemma. A proof of its one special 
case that $m\leq n-2$ can be found in \cite[Lemma~2.4]{IS02}.
\begin{lem}\label{Lem:degreesum}
Let $G$ be a graph with $m$ edges. Then
$M(G)\leq m^2+m.$
\end{lem}

\begin{proof}
Let $G$ be an extremal graph with the maximum $M(G)$ and $n:=|G|$.
Let $V(G)=\{u_1,\ldots,u_{n}\}$, and $d_i:=d_G(u_i)$ for each $u_i\in V(G)$.
We may assume that $d_1\geq\cdots\geq d_{n}\geq1$.

If there exists some integer $i\geq2$ such that $u_iu_1\notin E(G)$,
then we choose a vertex $u_j\in N_G(u_i)$ and define
$G':=G-u_iu_j+u_iu_1$.
Now $d_{G'}(u_1)=d_1+1,$ $d_{G'}(u_j)=d_j-1$ and $d_{G'}(u_k)=d_G(u_k)$ for
each $k\in\{2,\ldots,n\}\setminus\{j\}.$
Consequently,
\begin{eqnarray*}
M(G')-M(G)=(d_1+1)^2+(d_j-1)^2-d_1^2-d_j^2=2d_1-2d_j+2\geq2,
\end{eqnarray*}
a contradiction. Thus, $N_G(u_1)=V(G)\setminus\{u_1\}$, 
and so $d_1=n-1$.

Now let $e(G-u_1)$ be the number of edges in $G-u_1$.
Clearly, $e(G-u_1)=m-d_1$.
If $e(G-u_1)=0$, then $G\cong K_{1,m}$,
and so $\sum_{i=1}^{n}d^2_i=m^2+m$, as desired.
In the following, we assume $e(G-u_1)\geq1$.

If $e(G-u_1)\leq d_1-2$,
then $d_i+d_j\leq e(G-u_1)+3\leq d_1+1$ for each $u_iu_j\in E(G-u_1)$.
Now let $G'=G-u_iu_j+u_1u_0$, where $u_iu_j\in E(G-u_1)$ and
$u_0$ is a new vertex adjacent only to $u_1$ in $G'$.
Then
\begin{eqnarray*}
M(G')-M(G)\!\!&=&\!\! (d_1+1)^2+1+(d_i-1)^2+(d_j-1)^2-d_1^2-d_i^2-d_j^2\\
\!\!&=&\!\! 2(d_1-d_i-d_j)+4.
\end{eqnarray*}
It follows that $M(G')>M(G)$, a contradiction.
Therefore, $e(G-u_1)\geq d_1-1$.

Now let $e(G-u_1)=k$ and define a new graph $G':=K_{1,d_1+k}$.
Then $k\geq d_1-1$ and $e(G')=d_1+k=e(G)=m$.
Note that $n=d_1+1$ and $2k=2e(G-u_1)=\sum_{i=2}^{n}(d_i-1)$.
Hence, $2kd_1\geq \sum_{i=2}^{n}d_i^2-d_1^2=M(G)-2d_1^2.$
It follows that
\begin{eqnarray*}
M(G')-M(G)=
(k+d_1)^2+(k+d_1)-M(G)
\geq k^2-d_1^2+(k+d_1)\geq0,
\end{eqnarray*}
as $k\geq d_1-1$.
Thus, $M(G)\leq M(G')=m^2+m$.
This proves Lemma \ref{Lem:degreesum}.
\end{proof}

\section{Proof of Theorem \ref{Thm:Mainresult}}\label{Sec:Mainproof}

In this section, we give a proof of Theorem \ref{Thm:Mainresult}.
We would like to point out that the techniques used in the left part
are completely different from \cite{NZ21+}.

\subsection{A key lemma}
We first prove a key lemma.

\begin{lem}\label{lem7}
Let $G$ be a graph of size $m\geq1.8\times10^9$ and
$X$ be the
Perron vector of $G$ with component $x_u$ corresponding to $u\in V(G)$.
If $\lambda(G)\ge
\sqrt{m}$ and $x_ux_v>\frac1{9\sqrt{m}}$ for any $uv\in E(G)$, then $f(G)\ge \frac{m^2}{500}$ unless
$G$ is a star (possibly with some isolated vertices).
\end{lem}

\begin{proof}
We may assume that $\delta(G)\geq1$, where $\delta(G)$ is the minimum degree of $G$.
Then $G$ is connected (otherwise, we can find an edge $uv$ with $x_ux_v=0$).
By Perron-Frobenius theorem, $X$ is a positive vector.
Let $A=\{u\in V(G): x_u>\frac1{3\sqrt[4]{m}}\}$ and $B=V(G)\setminus A.$
Clearly, $B$ is an independent set.
Now suppose that $f(G)<\frac{m^2}{500}$
and set $\lambda:=\lambda(G)$.
We will prove a series of claims.

\begin{claim}\label{cl1}
We have $\delta(G)\geq2$ unless $G\cong K_{1,m}$.
\end{claim}

\begin{proof}
Assume that there exists a vertex $u\in V(G)$ with $d_G(u)=1$ and $N_G(u)=\{\bar{u}\}$.
Then $x_ux_{\bar{u}}=\frac{x_{\bar{u}}^2}\lambda\leq \frac{x_{\bar{u}}^2}{\sqrt{m}}.$
Since $x_ux_{\bar{u}}>\frac1{9\sqrt{m}}$, we have $x_{\bar{u}}>\frac13.$
Let $u^*\in V(G)$ with $x_{u^*}=\max_{v\in V(G)}x_v$.
Then $x_{u^*}>\frac13$.

Now let $S:=N_G(u^*)$, $T:=V(G)\setminus (S\cup\{u^*\})$, and
$N_S(v)=N_G(v)\cap S$ for a vertex $v\in V(G)$.
Moreover, we partite $S$ into three subsets $S_1$, $S_2$ and $S_3$,
where $S_1=\{v: \frac14<x_v\leq x_{u^*}\},$ $S_2=\{v: \frac16<x_v\leq \frac14\}$,
and $S_3=\{v: 0<x_v\leq \frac16\}$.

Choose a vertex $u\in S_1$ arbitrarily.
By Cauchy-Schwarz inequality,
\begin{eqnarray}\label{eq3}
(\lambda x_u)^2=\Big(\sum_{v\in N_G(u)}x_v\Big)^2\leq d_G(u)\sum_{v\in N_G(u)}x_v^2\leq d_G(u)(1-x_u^2).
\end{eqnarray}
Since $x_u>\frac14$ and $\lambda\geq\sqrt{m}$, we have $d_G(u)\geq \frac m{15}.$
If $|N_T(u)|\leq \frac{m}{450}$, then $|N_S(u)|\geq\frac m{15}-\frac{m}{450}-1\geq \frac{m}{15.6}$,
and thus $G$ contains a copy of $K_{2,\lceil\frac m{15.6}\rceil}$.
Hence, $G$ contains at least ${\lceil\frac m{15.6}\rceil\choose 2}$ ($\geq \frac{m^2}{500}$)
quadrilaterals, a contradiction.
Therefore, $|N_T(u)|\geq \frac{m}{450}$ and $|N_S(u)|<\frac{m}{15.6}$.
Now let $S^*=\{v\in S: x_v<\frac1 {108}\}$, $T^*=\{v\in T: x_v<\frac1 {108}\}$ and
$V'=(S\setminus S^*)\cup(T\setminus T^*)$.
Since $X$ is a unit vector, we have $|V'|\leq 108^2.$
By Cauchy-Schwarz inequality,
\begin{eqnarray}\label{eq0}
\sum_{v\in V'}x_v
\leq \sqrt{|V'|\sum_{v\in V'}x_v^2}\leq \sqrt{|V'|}\leq108.
\end{eqnarray}
Consequently,
$$\sum_{v\in N_S(u)}x_v=\sum_{v\in N_{S\setminus S^*}(u)}x_v+\sum_{v\in N_{S^*}(u)}x_v
\leq 108+\frac 1{108}|N_{S^*}(u)|.$$
Recall that $|N_{S^*}(u)|\leq |N_{S}(u)|\leq\frac{m}{15.6}$ and $x_{u^*}>\frac13$.
It follows that
\begin{eqnarray*}
\sum_{v\in N_S(u)}x_v\leq (324+\frac 1{36}|N_{S^*}(u)|)x_{u^*}
<\frac 1{36}\cdot\frac m{15}x_{u^*}.
\end{eqnarray*}
On the other hand, note that $|N_{T^*}(u)|\geq |N_T(u)|-108^2\geq\frac{m}{525}$
and $x_v<\frac1{108}<\frac1{36}x_{u^*}$ for any $v\in T^*$. Then
\begin{eqnarray*}
\sum_{v\in N_T(u)}x_v
&<&\sum_{v\in N_{T\setminus T^*}(u)}x_{u^*}+\sum_{v\in N_{T^*}(u)}\frac 1{36}x_{u^*}
\leq|N_T(u)|x_{u^*}-\frac {35}{36}\cdot\frac{m}{525}x_{u^*}\\
&=&|N_T(u)|x_{u^*}- \frac 1{36}\cdot\frac m{15}x_{u^*}.
\end{eqnarray*}
It follows that $\sum_{v\in N_{S\cup T}(u)}x_v<|N_T(u)|x_{u^*}.$
Let $e(S,T)$ be the number of edges from $S$ to $T$,
and $e(S)$ be the number of edges within $S$. Then
\begin{eqnarray}\label{eq4}
\sum_{u\in S_1}\sum_{v\in N_{S\cup T}(u)}x_v<e(S_1,T)x_{u^*}.
\end{eqnarray}

Secondly, consider a vertex $u\in S_2$ arbitrarily.
Note that $x_u>\frac16$ and $\lambda\geq\sqrt{m}$.
Then (\ref{eq3}) gives
$d_G(u)\geq \frac m{35}.$
Since $S^*\subseteq S_3$ and $x_{u^*}-x_u>\frac13-\frac14=\frac1{12}$,
we have $$\sum_{v\in N_{S^*}(u)}x_v\leq \frac 1{108}|N_{S_3}(u)|
\leq\frac19|N_{S_3}(u)|(x_{u^*}-x_u),$$
and by (\ref{eq0}) we have
$\sum_{v\in N_{S\setminus S^*}(u)}x_v\leq \sum_{v\in V'}x_v\leq108$. Then
\begin{eqnarray}\label{eq5}
\sum_{v\in N_S(u)}\!\!x_v=\sum_{v\in N_{S\setminus S^*}(u)}x_v+\!\!\!\!\sum_{v\in N_{S^*}(u)}\!\!x_v
\leq 108+\frac 19|N_{S_3}(u)|(x_{u^*}-x_u).
\end{eqnarray}
If $|N_{S^*}(u)|\geq \frac m{72}$,
then $|N_{S_3}(u)|\geq |N_{S^*}(u)|\geq\frac m{72}$.
Since $x_{u^*}-x_u>\frac1{12}$, it follows from (\ref{eq5}) that
$\sum_{v\in N_S(u)}x_v<|N_{S_3}(u)|(x_{u^*}-x_u)$, and thus
\begin{eqnarray}\label{eq6}
\sum_{v\in N_{S\cup T}(u)}x_v<|N_{S_3}(u)|(x_{u^*}-x_u)+|N_T(u)|x_{u^*}.
\end{eqnarray}
If $|N_{S^*}(u)|\leq\frac m{72}$, then 
$$|N_{T^*}(u)|\geq d_G(u)-|N_{S^*}(u)|-108^2>\frac m{72}.$$
Hence,
\begin{eqnarray*}
\sum_{v\in N_{T^*}(u)}x_v\leq|N_{T^*}(u)|\cdot\frac1{108}<
|N_{T^*}(u)|x_{u^*}-108,
\end{eqnarray*}
as $x_{u^*}>\frac13.$
It follows that $\sum_{v\in N_T(u)}x_v<|N_T(u)|x_{u^*}-108.$
Combining with (\ref{eq5}), we can also obtain (\ref{eq6}).
Therefore, in both cases we have
\begin{eqnarray}\label{eq7}
\sum_{u\in S_2}\sum_{v\in N_{S\cup T}(u)}x_v<e(S_2,S_3)x_{u^*}+e(S_2,T)x_{u^*}-\sum_{u\in S_2}|N_{S_3}(u)|x_u.
\end{eqnarray}

Thirdly, we consider an arbitrary vertex $u\in S_3$.
Since $x_{u^*}>\frac13$, we have $x_v\leq \frac16<\frac12x_{u^*}$ for each $v\in N_{S_3}(u)$.
Thus, $\sum_{u\in S_3}\sum_{v\in N_{S_3}(u)}x_v\leq e(S_3)x_{u^*}$,
with equality if and only if $e(S_3)=0$.
Therefore,
\begin{eqnarray}\label{eq8}
\sum_{u\in S_3}\sum_{v\in N_{S\cup T}(u)}x_v\leq e(S_3,S_1)x_{u^*}
+e(S_3)x_{u^*}+e(S_3,T)x_{u^*}+\sum_{u\in S_3}\sum_{v\in N_{S_2}(u)}x_v.
\end{eqnarray}

Notice that $$\sum_{u\in S_2}|N_{S_3}(u)|x_u=\sum_{u\in S_3}\sum_{v\in N_{S_2}(u)}x_v.$$
Combining with (\ref{eq4}), (\ref{eq7}) and (\ref{eq8}), we have
\begin{eqnarray}\label{eq9}
\sum_{u\in S}\sum_{v\in N_{S\cup T}(u)}x_v\leq (e(S)
+e(S,T))x_{u^*},
\end{eqnarray}
where if equality holds then $S_1\cup S_2=\varnothing$ and $e(S_3)=0$,
that is, $e(S)=0$.
Furthermore,
we can see that
\begin{eqnarray*}
\lambda^2x_{u^*}
=\sum_{u\in S}\sum_{v\in N_G(u)}\!\!\!x_v=|S|x_{u^*}+\sum_{u\in S}\sum_{v\in N_{S\cup T}(u)}\!\!\!x_v
\leq(|S|+e(S)+e(S,T))x_{u^*}\leq mx_{u^*}.
\end{eqnarray*}
Since $\lambda\geq \sqrt{m}$, the above inequality holds in equality, that is, $\lambda=\sqrt{m}$.
Therefore, $m=|S|+e(S)+e(S,T)$, and (\ref{eq9}) holds in equality (hence $e(S)=0$).
This implies that $G$ is a bipartite graph.
By Lemma \ref{lem5}, $G$ is a complete bipartite graph.
Since $f(G)<\frac{m^2}{500}$, $G$ can only be a star.
This completes the proof.
\end{proof}

In the following, we may assume that $G\ncong K_{1,m}$.
Then by Claim \ref{cl1}, $\delta(G)\geq2$.

\begin{claim}\label{cl2}
$|A|\leq9\sqrt{m}$.
\end{claim}

\begin{proof}
Recall that $x_u>\frac1{3\sqrt[4]{m}}$ for each $u\in A$.
Thus $\sum_{u\in A}x_{u}^2>\frac{|A|}{9\sqrt{m}}$, and hence
$|A|\leq 9\sqrt{m}\sum_{u\in A}x_{u}^2\leq9\sqrt{m}.$
\end{proof}

\begin{claim}\label{cl3}
Let $|G|=\frac m2+b$. Then $-\frac m{125}\leq b\leq|A|$.
\end{claim}

\begin{proof}
Set $\lambda':=\lambda_{|G|}$. Note that $\lambda\geq \sqrt{m}$.
By Lemmas \ref{lem2} and \ref{lem3},
\begin{eqnarray}\label{eq10}
f(G)\geq\frac18(\lambda^4+\lambda'^4)-\frac14M(G)
\geq\frac18(\lambda^4+\lambda'^4)-\frac{|G|}4\lambda^2\geq\frac18\lambda'^4-\frac b4\lambda^2.
\end{eqnarray}
If $b<-\frac m{125}$, then $$f(G)\geq-\frac b4\lambda^2\geq\frac{m}{500}\lambda^2\geq\frac{m^2}{500},$$
a contradiction.
Thus, $b\geq-\frac m{125}$.

On the other hand, recall that $e(B)=0$ and $\delta(G)\geq2$,
then
$$m\geq e(B,A)\geq 2|B|=2(|G|-|A|)=2(\frac m2+b-|A|).$$
Thus, $b\leq |A|$, as desired.
\end{proof}

\begin{claim}\label{cl4}
$\Delta(G)\leq \frac2{15}m,$ where $\Delta(G)$ is the maximum degree of $G$.
\end{claim}

\begin{proof}
We know that $\sum_{i=1}^{|G|}\lambda_i^2=2m$. Thus, $\lambda^2=\lambda_1^2\leq 2m$.
Combining with (\ref{eq10}) and $b\leq|A|\leq9\sqrt{m}$, we have
\begin{eqnarray}\label{eq11}
f(G)\geq\frac18\lambda'^4-\frac b4\lambda^2\geq\frac18\lambda'^4-9m^{\frac32}.
\end{eqnarray}

Now if there exists some $u\in V(G)$ with $d_G(u)>\frac2{15}m$, then
$$|N_B(u)|\geq d_G(u)-|A|>\frac2{15}m-9\sqrt{m}.$$
Since $e(B)=0$, $G$ contains $K_{1,|N_B(u)|}$ as an induced subgraph.
By Lemma \ref{Lem:CauchyInter},
$$\lambda'\leq -\sqrt{|N_B(u)|}<-\sqrt{\frac2{15}m-9\sqrt{m}},$$
and by (\ref{eq11}) we have
\begin{eqnarray*}
f(G)\geq\frac18\lambda'^4-9m^{\frac32}\geq \frac18\Big(\frac2{15}m-9\sqrt{m}\Big)^2-9m^{\frac32}>\frac{m^2}{500},
\end{eqnarray*}
for $m\geq 1.8\times10^9$.
We have a contradiction.
Therefore, $\Delta(G)\leq \frac2{15}m.$
\end{proof}

\begin{claim}\label{cl5}
Let $B^*=\{u\in V(G): d_G(u)=2\}$. Then $B^*\subseteq B$ and
\begin{eqnarray}\label{eq12}
\frac m2+3(b-|A|)\leq |B^*|\leq \frac m2.
\end{eqnarray}
\end{claim}

\begin{proof}
Let $u\in B^*$ and $N_G(u)=\{u_1,u_2\}$. Then $\lambda x_{u}=x_{u_1}+x_{u_2}\leq 2$.
Since $\lambda\geq \sqrt{m}$, we have $x_u\leq\frac{2}{\sqrt{m}}<\frac 1{3\sqrt[4]{m}},$
and so $u\in B$.

Recall that $e(B)=0$. Thus, $e(B^*)=0$, and $m\geq e(B^*,A)\geq 2|B^*|$.
This gives $|B^*|\leq \frac m2.$
On the other hand, note that $|B|=|G|-|A|=\frac m2+b-|A|$,
then
\begin{eqnarray*}
m\geq e(B,A)\geq2|B^*|+3(|B|-|B^*|)=\frac32m+3(b-|A|)-|B^*|.
\end{eqnarray*}
It follows that $|B^*|\geq\frac m2+3(b-|A|).$
\end{proof}

\begin{claim}\label{cl6}
Let $A^*=\{v\in N_G(u): u\in B^*\}$. Then $A^*\subseteq A$ and $|A^*|\leq24$.
\end{claim}

\begin{proof}
Since $e(B)=0$, we have $N_G(u)\subseteq A$ for any $u\in B^*$.
Thus, $A^*\subseteq A$.
Furthermore, we will see that $\frac1{25}<x_v^2\leq\frac{2}{17}$ for each $v\in A^*$.

Let $v$ be an arbitrary vertex in $A^*$.
By Cauchy-Schwarz inequality,
\begin{eqnarray*}
(\lambda x_v)^2=\Big(\sum_{u\in N_G(v)}x_u\Big)^2
\leq d_G(v)\sum_{u\in N_G(v)}x_u^2\leq d_G(v)(1-x_v^2)
\leq \frac2{15}m(1-x_v^2),
\end{eqnarray*}
as $\Delta(G)\leq \frac 2{15}m$.
Since $\lambda\geq\sqrt{m}$, we have $x_v^2\leq \frac2{17}.$

If there exists a vertex $v\in A^*$ with $x_v^2\leq\frac1{25}$,
then by the definition of $A^*$, we can find a vertex $u\in N_{B^*}(v)$.
Clearly, $$\lambda x_u\leq x_v+\sqrt{\frac2{17}}\leq\frac15+\sqrt{\frac2{17}}<\frac59.$$
Consequently,
$$x_ux_v<\frac 1\lambda\cdot\frac 59\cdot\frac15\leq \frac1{9\sqrt{m}},$$
which contradicts the condition of Lemma \ref{lem7}.
Therefore, $x_v^2>\frac1{25}$ for any $v\in A^*$, and so $|A^*|\leq24.$
\end{proof}

\begin{claim}\label{cl7}
Let $V'':=(A\setminus A^*)\cup (B\setminus B^*)$.
Then $|V''|\leq\frac m{60}$ and $e(V'')\leq \frac m{20}.$
\end{claim}

\begin{proof}
Recall that $|A\cup B|=|G|=\frac m2+b$.
Combining with (\ref{eq12}), we obtain that
$|V''|\leq|G|-|B^*|\leq 3|A|-2b.$
Moreover, by Claims \ref{cl2} and \ref{cl3}, we have
$|A|\leq9\sqrt{m}$ and $b\geq-\frac m{125}$. Thus,
$|V''|\leq 27\sqrt{m}+\frac 2{125}m\leq\frac m{60}.$

Now we estimate $e(V'')$.
Again by $|A|\leq9\sqrt{m}$, $b\geq-\frac m{125}$ and (\ref{eq12}), we have
\begin{eqnarray*}
e(A^*,B^*)=2|B^*|\geq m+6(b-|A|)\geq m-\frac6{125}m-54\sqrt{m}.
\end{eqnarray*}
It follows that
$e(V'')\leq m-e(A^*,B^*)\leq\frac6{125}m+54\sqrt{m}\leq \frac m{20}.$
\end{proof}

Now we give the final proof of Lemma \ref{lem7}.
For convenience, let $d'(u):=|N_{V''}(u)|$ for each $u\in V''$.
Note that $e(V'',B^*)=0$. Thus by Claim \ref{cl6}, $$d_G(u)\leq d'(u)+|A^*|\leq d'(u)+24$$
for each vertex $u\in V''$.
Consequently,
\begin{eqnarray}\label{eq13}
\sum_{u\in V''}d^2_G(u)\leq \sum_{u\in V''}(d'(u)+24)^2=96e(V'')+24^2|V''|+\sum_{u\in V''}d'^2(u).
\end{eqnarray}
Since $e(V'')\leq \frac m{20}$, by Lemma \ref{Lem:degreesum} we have
$\sum_{u\in V''}d'^2(u)\leq \frac{m^2}{400}+\frac m{20}.$
Combining this with Claim \ref{cl7} and (\ref{eq13}), we have
\begin{eqnarray}\label{eq14}
\sum_{u\in V''}d^2_G(u)
\leq 96\cdot\frac m{20}+24^2\cdot\frac m{60}+\frac{m^2}{400}+\frac m{20}<\frac {m^2}{225}.
\end{eqnarray}
On the other hand, by Claim \ref{cl4}, $\Delta(G)\leq \frac 2{15}m$, and so
$$\sum_{u\in A^*}d^2_G(u)\leq|A^*|(\Delta(G))^2\leq \frac {96}{225}m^2$$ (as $|A^*|\leq24$).
Moreover, by Claim \ref{cl5} $|B^*|\leq \frac m2$, and thus
$$\sum_{u\in B^*}d^2_G(u)=4|B^*|\leq 2m.$$
Combining with (\ref{eq14}), we get
\begin{eqnarray*}
M(G)=\sum_{u\in V''\cup A^*\cup B^*}d^2_G(u)\leq \frac{1}{225}m^2+\frac{96}{225}m^2+2m<\frac{100}{225}m^2=\frac49m^2.
\end{eqnarray*}
Now by Lemma \ref{lem3}, we have
\begin{eqnarray*}
f(G)\geq \frac18\lambda^4-\frac14M(G)\geq \frac18m^2-\frac1{9}m^2=\frac1{72}m^2>\frac1{500}m^2,
\end{eqnarray*}
a contradiction.
This completes the proof.
\end{proof}

\subsection{Nikiforov's deleting small eigenvalue edge method}
Over the past decades, Nikiforov developed some novel tools and techniques for
solving problems in spectral graph theory (see \cite{N11}). One is the method we called ``deleting
small eigenvalue edge method", or ``The DSEE Method". Generally speaking, an edge $xy\in E(G)$
is called a \emph{small eigenvalue edge}, if $x_ux_v$ is small where $x_u,x_v$ are
Perron components.

By using this method, Nikiforov \cite{N09} successfully proved the following results, of which
some original ideas appeared in \cite{N11} earlier:
\begin{itemize}
\item Every graph on $m$ edges contains a 4-cycle if $\lambda(G)\geq\sqrt{m}$ and $m\geq10$,
unless it is a star with possibly some isolated vertices (see Claim 4 in \cite[pp.~2903]{N09});

\item Every graph on $m$ edges satisfies that the booksize $bk(G)>\frac{\sqrt[4]{m}}{12}$ if $\lambda(G)\geq\sqrt{m}$,
unless it is a complete bipartite graph with possibly some isolated vertices
(see \cite{N21}, this confirmed a conjecture in \cite{ZLS21}).
\end{itemize}

One main ingredient in the proof of Theorem \ref{Thm:Mainresult} is using this method.
\subsection{Proof of Theorem \ref{Thm:Mainresult}}
Now we are ready to give the proof of Theorem \ref{Thm:Mainresult}.

\vspace{2mm}
\noindent
{\bf Proof of Theorem \ref{Thm:Mainresult}.}
Let $G$ be a graph with $e(G)=m$ and $\lambda(G)\geq\sqrt{m}.$
By using the Nikiforov DESS Method \cite{N21},
we first construct a sequence of graphs.

\vspace{2mm}
\noindent
(i) Set $i:=0$ and $G_0:=G$.\\
(ii) If $i=\lfloor\frac m2\rfloor,$ stop.\\
(iii) Let $X=(x_1,x_2\ldots,x_{|G_i|})^T$ be the Perron vector of $G_i$.\\
(iv) If there exists $uv\in E(G_i)$ with $x_ux_v\leq\frac1{9\sqrt{e(G_i)}}$,
set $G_{i+1}:=G_i-uv$ and $i:=i+1$.\\
(v) If there is no such edge, stop.

\vspace{2mm}

Assume that $G_k$ is the resulting graph of the graph sequence
constructed by the above algorithm.
Then $k\leq\lfloor\frac m2\rfloor.$
We can obtain the following two claims.

\begin{claim}\label{cl8}
$\lambda(G_{i+1})\geq\sqrt{m-i-1}$
for each $i\in \{0,1,\ldots,k-1\}$.
\end{claim}

\begin{proof}
Let $X$ be the Perron vector of $G_i$ with component $x_u$ corresponding to $u\in V(G_i)$.
Then, there exists $uv\in E(G_i)$ with $x_ux_v\leq\frac1{9\sqrt{e(G_i)}}$.
Thus,
$$\lambda(G_{i+1})\geq X^TA(G_{i+1})X=X^TA(G_{i})X-2x_ux_v\geq\lambda(G_i)-\frac2{9\sqrt{e(G_i)}}.$$
Hence,
$$\lambda(G_0)\leq \lambda(G_1)+\frac2{9\sqrt{m}}\leq\cdots
\leq \lambda(G_{i+1})+\sum_{j=0}^{i}\frac2{9\sqrt{m-j}}.$$
It follows that
\begin{eqnarray}\label{eq15}
\lambda(G_{i+1})\geq\lambda(G_0)-\frac{2(i+1)}{9\sqrt{m-i-1}}\geq\sqrt{m}-\frac{2(i+1)}{9\sqrt{m-i-1}}.
\end{eqnarray}
This implies that $\lambda(G_{i+1})\geq\sqrt{m-i-1}$, as $i+1\leq k\leq\lfloor\frac m2\rfloor$.
\end{proof}

Now we may assume that all isolated vertices are removed from each $G_i$,
where $i\in \{0,1,\ldots,k\}$.

\begin{claim}\label{cl9}
$G_k$ cannot be a star unless $G_k=G_0\cong K_{1,m}$.
\end{claim}

\begin{proof}
Suppose to the contrary that $k\geq1$ while $G_k$ is a star.
Since $e(G_k)=m-k$, we have $G_k\cong K_{1,m-k}$.
Let $u_0$ be the central vertex of $G_k$ and $u_1,\ldots,u_{m-k}$ be the leaves.
We now let $G_k=G_{k-1}-uv$ and $X$ be the Perron vector of $G_{k-1}$.

If $uv$ is a pendent edge incident to $u_0$, say $uv=u_0u_{m-k+1}$,
then $$\lambda(G_{k-1})=\sqrt{e(G_{k-1})}=\sqrt{m-k+1}$$ and
$\lambda(G_{k-1})x_{u_i}=x_{u_0}$ for $i\in\{1,2,\ldots,m-k+1\}$.
Hence, $\|X\|_2=\sum_{i=0}^{m-k+1}x_{u_i}^2=2x_{u_0}^2$, which gives
$x_{u_0}^2=\frac12$.
It follows that $$x_{u_0}x_{u_{m-k+1}}=\frac{x^2_{u_0}}{\sqrt{e(G_{k-1})}}>\frac{1}{9\sqrt{e(G_{k-1})}},$$
which contradicts the definition of $G_k$.

If $uv$ is an isolated edge or a pendent edge not incident to $u_0$,
then $G_{k-1}$ is bipartite but not complete bipartite.
By Lemma \ref{lem5}, $\lambda(G_{k-1})<\sqrt{e(G_{k-1})}$,
which contradicts Claim \ref{cl8}.

Now we conclude that $uv$ is an edge within $V(G_k)\setminus\{u_0\}$,
say $uv=u_1u_2$,
then $x_{u_1}=x_{u_2}$ and  $\lambda(G_{k-1})x_{u_1}=x_{u_0}+x_{u_2}$.
Hence, $x_{u_1}=\frac{x_{u_0}}{\lambda(G_{k-1})-1}<\frac12x_{u_0},$
as $\lambda(G_{k-1})\geq\sqrt{m-k+1}$ by Claim \ref{cl8}.
Consequently,
$$\lambda^2(G_{k-1})x_{u_0}=\sum_{i=1}^{m-k}\lambda(G_{k-1})x_{u_i}
=(m-k)x_{u_0}+(x_{u_1}+x_{u_2})<(m-k+1)x_{u_0}.$$
It follows that $\lambda(G_{k-1})<\sqrt{m-k+1}$,
which also contradicts Claim \ref{cl8}.
\end{proof}

Now we finish the final proof of Theorem \ref{Thm:Mainresult}.
Assume that $G$ is not a star. Then $G_k$ is not a star by Claim \ref{cl9};
moreover, $\lambda(G_k)\geq\sqrt{m-k}=\sqrt{e(G_{k})}$ by Claim \ref{cl8}.
If $k<\lfloor\frac m2\rfloor$, then
$x_ux_v>\frac1{9\sqrt{e(G_k)}}$ for any edge $uv\in E(G_k)$.
Since $e(G_k)=m-k>\frac m2,$
by Lemma \ref{lem7} $f(G_k)\geq \frac {((e(G_k))^2}{500}>\frac {m^2}{2000},$
and so $f(G)>\frac {m^2}{2000},$
as desired.

If $k=\lfloor\frac m2\rfloor$,
then by (\ref{eq15}) we have
\begin{eqnarray*}
\lambda(G_k)\geq\sqrt{m}-\frac{2k}{9\sqrt{m-k}}\geq \sqrt{m}-\frac{m}{9\sqrt{\frac m2}}=\Big(1-\frac{\sqrt{2}}9\Big)\sqrt{m},
\end{eqnarray*}
and so $$\lambda^4(G_k)\geq(1-\frac{\sqrt{2}}9)^4m^2=0.5047m^2>0.504m^2+4m.$$
On the other hand, by Lemma \ref{Lem:degreesum},
$$M(G_k)\leq (e(G_k))^2+e(G_k)=\lceil\frac m2\rceil^2+\lceil\frac m2\rceil\leq 0.25m^2+2m.$$
Thus by Lemma \ref{lem3},
\begin{eqnarray*}
f(G_k)\geq \frac18\lambda^4(G_k)-\frac14M(G_k)>\frac18(0.504-0.5)m^2=\frac1{2000}m^2,
\end{eqnarray*}
and so $f(G)>\frac {m^2}{2000}.$
This completes the proof. $\hfill\blacksquare$

\section{Concluding remarks}
We do not try our best to optimize the constant ``$\frac{1}{2000}$"
in Theorem \ref{Thm:Mainresult}. So it is natural to pose the following problem:
\begin{prob}
Determine $\lim\limits_{m\rightarrow \infty} \frac{f(m)}{m^2}$.
(We think that the upper bound in Proposition \ref{Prop} is close to the truth.)
\end{prob}

By Theorem \ref{Thm:Mainresult} and an inequality $\lambda(G)\geq \frac{2m}n$ due to Collatz and Sinogowitz \cite{CS57},
we deduce the following.
\begin{thm}
Let $G$ be a graph on $n$ vertices and $m$ edges. If
$m>\max\{\frac{n^2}{4},3.6\times 10^9\}$,
then $G$ contains $\frac{n^4}{32000}$
copies of 4-cycles.
\end{thm}

On the other hand, we would like to mention the following conjecture.
\begin{conj}[Conjecture 5.1 in \!\!\cite{ZLS21}]
Let $k\geq 2$ be a fixed positive integer and $G$ be a graph of sufficiently large size $m$ without
isolated vertices. If $\lambda(G)\geq \frac{k-1+\sqrt{4m-k^2+1}}{2}$, then
$G$ contains a cycle of length $t$ for every $t\leq 2k+2$,
unless $G=S_{\frac{m}{k}+\frac{k+1}{2},k}$.
\end{conj}

When $k=1$, the above conjecture reduces to Nikiforov's result (Theorem \ref{Thm:N09}).
Let $B_{r,k}$ be the join of an $r$-clique with an independent set of size $k$.
We conclude this note with a
new conjecture appeared in \cite{LFL} which extends Theorem \ref{Thm:N09}.
\begin{conj}[Conjecture 1.20 in \!\!\cite{LFL}]
Let m be large enough and G be a $B_{r,k}$-free graph with m edges. Then
$\lambda(G)\leq \sqrt{(1-\frac{1}{r})2m}$, with equality if and only if $G$ is a complete bipartite
graph for $r=2$, and $G$ is a complete regular $r$-partite graph for
$r\geq 3$ with possibly some isolated vertices.
\end{conj}



\begin{thebibliography}{99}
\setlength{\itemsep}{0pt}

\bibitem{AS08}
N. Alon, J.H. Spencer, The probabilistic method. Third edition.
With an appendix on the life and work of Paul Erd\H{o}s. Wiley-Interscience
Series in Discrete Mathematics and Optimization. John Wiley \& Sons,
Inc., Hoboken, NJ, 2008.

\bibitem{BG08}
L. Babai, B. Guiduli, Spectral extrema for graphs: the Zarankiewicz problem,
\emph{Electronic J. Combin.} \textbf{16} (2009), no. 1, R123, 8 pp.

\bibitem{BFP08}
A. Bhattacharya, S. Friedland, U.N. Peled, On the first eigenvalue of bipartite graphs, \emph{Electron. J.
Combin.} \textbf{15} (2008), no. 1, R144, 23 pp.

\bibitem{BH12}
A. Brouwer, W.H. Haemers, Spectra of Graphs, Springer, New York, 2012.

\bibitem{CG98}
F. Chung, R. Graham, Erd\H{o}s on Graphs.
His legacy of unsolved problems. A K Peters, Ltd., Wellesley, MA, 1998.

\bibitem{CS57}
L. Collatz, U. Sinogowitz,
Spektren endlicher Grafen,
\emph{Abh. Math. Semin. Univ. Hamb.},
{\bf 21} (1957), 63--77.

\bibitem{E38}
P. Erd\H{o}s, On sequences of integers no one of which divides the product of two others and some
related problems, \emph{Izvestiya Naustno-Issl. Inst. Mat. i Meh. Tomsk} 2 (1938) 74--82.

\bibitem{ES84}
P. Erd\H{o}s, M. Simonovits, Cube-supersaturated graphs and related problems, P
rogress in graph theory (Waterloo, Ont., 1982), 203--218, Academic Press, Toronto, ON, 1984.

\bibitem{H88}
M. Hofmeister, Spectral radius and degree sequence, \emph{Math. Nachr.} \textbf{139} (1988) 37--44.

\bibitem{IS02}
D. Ismailescu, D. Stefanica, Minimizer graphs for a class of extremal problems,
\emph{J. Graph Theory} {\bf 39} (2002), no. 4, 230--240.

\bibitem{LFL}
Y.T. Li, L.H. Feng, W.J. Liu, A survey on spectral conditions for some
extremal graph problem, \emph{Advances in Math.} (China) (2022), to appear.

\bibitem{LL09}
B.-L. Liu, M.-H. Liu, On the spread of the spectrum of a graph, \emph{Discrete Math.}
\textbf{309} (2009) 2727--2732.

\bibitem{LS83}
L. Lov\'{a}sz, M. Simonovits, On the number of complete subgraphs of a graph, II,
in: Studies in Pure Math, Birkh\"{a}user, 1983, 459--495.

\bibitem{N09}
V. Nikiforov, The maximum spectral radius of $C_4$-free graphs of given order and
size, \emph{Linear Algebra Appl.}
\textbf{430} (2009) 2898--2905.

\bibitem{N10}
V. Nikiforov,
A contribution to the Zarankiewicz problem,
\emph{Linear Algebra Appl.} {\bf 432} (2010), no. 6, 1405--1411.

\bibitem{N11}
V. Nikiforov, Some new results in extremal graph theory,
Surveys in combinatorics 2011, 141--181,
London Math. Soc. Lecture Note Ser., 392,
Cambridge Univ. Press, Cambridge, 2011.

\bibitem{N21}
V. Nikiforov, On a theorem of Nosal, arXiv:2104.12171 (2021).

\bibitem{N17}
B. Ning, On some papers of Nikiforov, \emph{Ars Combin.}
{\bf 135} (2017), 187--195.

\bibitem{NZ21+}
B. Ning, M.Q. Zhai, Counting substructrues and eigenvalues I: triangles,
submitted on Novmeber 2, 2021.

\bibitem{N70}
E. Nosal, Eigenvalues of Graphs, Master Thesis, University of Calgary, 1970.

\bibitem{Y21+}
T.C. Yang, Some extremal results on 4-cycles,
A talk given at the sixth Shanghai Jiaotong University workshop on graph
theory and combinatorics (2021).

\bibitem{ZLS21}
M.Q. Zhai, H.Q. Lin, J.L. Shu, Spectral extrema of graphs of fixed size: cyles and
complete bipartite graphs, \emph{European J. Combin.} \textbf{95} (2021) 103322, 18 pp.
\end{thebibliography}
\end{document}